\documentclass[a4paper,reqno]{amsart}
\usepackage[utf8]{inputenc}
\usepackage[T1]{fontenc}
\usepackage{lmodern}
\usepackage{amssymb,amsxtra}
\usepackage{nicefrac,mathtools,enumitem}
\setlist[enumerate,1]{label=\textup{(\arabic*)}}
\usepackage{mdwlist}

\usepackage{tikz}
\usetikzlibrary{matrix}
\tikzset{cd/.style=matrix of math nodes,row sep=2em,column sep=2em, text height=1.5ex, text depth=0.5ex}
\tikzset{cdar/.style=->,auto}
\tikzset{overar/.style={draw=white,double=black,double distance=.4pt,very thick}}

\usepackage{xcolor}

\usepackage{microtype}
\usepackage[pdftitle={Duality of braided quantum $E(2)$},
  pdfauthor={Atibur Rahaman},
  pdfsubject={Mathematics; MSC }
]{hyperref}

\usepackage[lite]{amsrefs}

\renewcommand{\PrintDOI}[1]{\href{http://dx.doi.org/\detokenize{#1}}{doi: \detokenize{#1}}}

\BibSpec{book}{%
  +{}  {\PrintPrimary}                {transition}
  +{,} { \textit}                     {title}
  +{.} { }                            {part}
  +{:} { \textit}                     {subtitle}
  +{,} { \PrintEdition}               {edition}
  +{}  { \PrintEditorsB}              {editor}
  +{,} { \PrintTranslatorsC}          {translator}
  +{,} { \PrintContributions}         {contribution}
  +{,} { }                            {series}
  +{,} { \voltext}                    {volume}
  +{,} { }                            {publisher}
  +{,} { }                            {organization}
  +{,} { }                            {address}
  +{,} { \PrintDateB}                 {date}
  +{,} { }                            {status}
  +{}  { \parenthesize}               {language}
  +{}  { \PrintTranslation}           {translation}
  +{;} { \PrintReprint}               {reprint}
  +{.} { }                            {note}
  +{.} {}                             {transition}
  +{} { \PrintDOI}                   {doi}
  +{} { available at \url}            {eprint}
  +{}  {\SentenceSpace \PrintReviews} {review}
}

\numberwithin{equation}{section}

\theoremstyle{plain}
\newtheorem{theorem}[equation]{Theorem}

\newtheorem{proposition}[equation]{Proposition}

\theoremstyle{definition}

\theoremstyle{remark}
\newtheorem{remark}[equation]{Remark}

\newcommand{\tenscorep}{\mathbin{\begin{tikzpicture}[baseline,x=.75ex,y=.75ex] \draw[line width=.2pt] (-0.8,1.15)--(0.8,1.15);\draw[line width=.2pt](0,-0.25)--(0,1.15); \draw[line width=.2pt] (0,0.75) circle [radius = 1];\end{tikzpicture}}}

\newcommand*{\Braiding}[2]{\begin{tikzpicture}[baseline]
    \draw[-] (0,0) -- (1.4ex,1.4ex) node[right,inner sep=0pt] {$\scriptstyle #2$};
    \draw[-,draw=white,line width=2.4pt] (0,1.4ex) -- (1.4ex,0);
    \draw[-] (1.4ex,0) -- (0,1.4ex) node[left,inner sep=0pt] {$\scriptstyle #1$};
  \end{tikzpicture}}
\newcommand*{\Dualbraiding}[2]{\begin{tikzpicture}[baseline]
    \draw[-] (1.4ex,0) -- (0,1.4ex) node[left,inner sep=0pt] {$\scriptstyle #1$};
    \draw[-,draw=white,line width=2.4pt] (0,0) -- (1.4ex,1.4ex);
    \draw[-] (0,0) -- (1.4ex,1.4ex) node[right,inner sep=0pt] {$\scriptstyle #2$};
  \end{tikzpicture}}

\newcommand*{\corep}[1]{\textup{#1}}          
\newcommand*{\Corep}[1]{\mathbb{#1}}          
\newcommand*{\DuCorep}[1]{\hat{\Corep{#1}}}   

\newcommand*{\dom}{\mathcal D}
\newcommand*{\Sp}{\text{Sp}}

\newcommand*{\nb}{\nobreakdash}
\newcommand*{\Star}{$^*$\nb-}

\newcommand*{\C}{\mathbb C}
\newcommand*{\Z}{\mathbb Z}

\newcommand*{\T}{\mathbb T}

\newcommand*{\G}[1][G]{\mathbb #1}
\newcommand*{\DuG}[1][G]{\widehat{\mathbb{#1}}}
\newcommand*{\qE}{\textup{E}_\textup {q}(2)}

\newcommand*{\Comult}[1][]{\Delta_{#1}}
\newcommand*{\DuComult}[1][]{{\Delta}_{\widehat{#1}}}

\newcommand*{\Mod}[1]{\abs{#1}}
\newcommand*{\Ph}[1]{\Phi_{#1}}

\newcommand*{\Bound}{\mathbb B}
\newcommand*{\Comp}{\mathbb K}

\newcommand*{\transpose}{\mathsf T}

\newcommand*{\Contvin}{\textup C_0}
\newcommand*{\Cont}{\textup C}
\newcommand*{\QuantExp}[1][]{F_{|q|}{#1}}

\newcommand*{\Mor}{\textup{Mor}}
\newcommand*{\Id}{\textup{id}}

\newcommand*{\Multunit}[1][]{\mathbb V}

\newcommand*{\BrMultunit}{\mathbb F}
\newcommand*{\DuBrMultunit}{\widehat{\BrMultunit}}

 \newcommand*{\dual}[1]{\widehat{#1}}

\newcommand*{\Rmat}{\textup R}
\newcommand*{\DuRmat}{\widehat{\textup R}}
\newcommand*{\Rmattxt}{R}

\newcommand*{\Flip}{\Sigma}
\newcommand*{\flip}{\sigma}
\newcommand*{\Cst}{\textup C^*}

\newcommand*{\Cstcat}{\mathfrak{C^*alg}}
\newcommand*{\Corepcat}{\mathfrak{Rep}}

\newcommand*{\Hils}[1][H]{\mathcal{#1}}
\newcommand*{\Mult}{\mathcal M}
\newcommand*{\U}{\mathcal U}

\newcommand*{\defeq}{\mathrel{\vcentcolon=}}

\newcommand*{\abs}[1]{\lvert#1\rvert}

\newcommand*{\conj}[1]{\overline{#1}}
\newcommand*{\cl}[1]{\overline{#1}}
\newcommand*{\Specn}{\cl{\C}^{|q|}}
\newcommand*{\Aff}{\mathrel{\eta}}

\DeclareMathOperator{\Aut}{Aut}
\DeclareMathOperator{\Hom}{Hom}

\begin{document}

\title[Dual braided quantum \(\textup{E}(2)\) groups]{Dual braided quantum \(\textup{E}(2)\) groups}

\author{Atibur Rahaman}
\email{atibur.pdf@iiserkol.ac.in}
\address{
 Department of Mathematics and Statistics, Indian Institute of Science Education and Research Kolkata, Mohanpur - 741246, West Bengal, India.}

\begin{abstract}
 	An explicit construction of the braided dual of quantum \(\textup{E}(2)\) groups is described 
	over the circle group \(\T\) with respect to a specific \(\Rmattxt\)\nb-matrix \(\Rmat\). 
	Furthermore, the corresponding bosonization is also described.
\end{abstract}


\maketitle

\section{Introduction}
  \label{sec:introduction}
	
%
	
  	Along with a clear formula for the bicharacter characterising the duality of quantum groups 
	in the operator algebraic context, a beautiful duality for the quantum \(\textup{E}(2)\) group for 
	non\nb-zero real deformation parameters was discovered in~\cite{W1991a}. 
	It was shown in~\cite{VW1996} that the bidual of quantum \(\textup{E}(2)\) group is 
	isomorphic to quantum \(\textup{E}(2)\) group itself. In this article we 
	expand this duality to accomodate the complex values of the deformation parameter. 
	
	The \(q\)-deformations of the (double cover of) \(\textup{E}(2)\) group, which Woronowicz 
	first described in~\cites{W1991b} for real deformation parameters, are classic examples of 
	locally compact quantum groups. These deformations were then extended in~\cite{RR2021} 
	for complex values of the deformation parameter \(q\) to produce braided locally compact 
	quantum \(\textup{E}(2)\) groups over the circle group \(\T\), viewed as a quasitriangular 
	quantum group with respect to a certain \(\Rmattxt\)\nb-matrix \(\Rmat\). 
	
	The braiding unitary was considered in the representation category of the above mentioned quasitriangular 
	quantum group \(\T\). Subsequently the braiding unitary induces a twisted monoidal structure on the 
	category~\(\Cstcat(\T)\), whose  objects are \(\Cst\)\nb-algebras with an action of \(\T\), 
	called \(\T\)\nb-\(\Cst\)\nb-algebras, and the morphisms 
	are the \(\T\)\nb-equivariant morphisms between two \(\T\)\nb-\(\Cst\)\nb-algebras. This category plays a 
	pivotal role in the construction  of both braided~\(\qE\) and its dual, henceforth denoted 
	by~\(\dual{\qE}\). However, unlike the case of braided~\(\qE\), the monoidal product 
	\(\boxtimes_{\DuRmat}\), in this case, is governed by the dual braiding induced by the 
	dual \(\Rmattxt\)\nb-matrix~\(\DuRmat\) in \(\Corepcat(\T)\), the representation category of \(\T\).
	Therefore, we provide a complete description of the dual of the braided quantum \(\textup{E}(2)\) 
	group building on the general construction of braided locally compact quantum groups over a 
	quasitriangular quantum group following~\cite{MRW2017}, taking in particular the quasitriangular 
	quantum group to be \(\T\), inside the general framework of braided multiplicative unitaries.  
	It turns out that the bidual of braided quantum \(\textup{E}(2)\)  group coincides with braided 
	quantum \(\textup{E}(2)\) group.

	Quantum \(\textup{E}(2)\) group is made to be a braided \(\Cst\)\nb-quantum group in an appropriate sense by the change 
	from real to complex values of the deformation parameter \(q\), as shown in~\cite{RR2021}, and the 
	associated multiplicative unitary then becomes a braided multiplicative 
	unitary~\cite{RR2021}*{Theorem 4.6}.
	The underlying \(\Cst\)\nb-algebra of \(\qE\) group is generated by a unitary \(v\) and an 
	unbounded normal element \(n\) with spectrum contained in the set defined and denoted 
	by~\(\Specn\defeq\big\{\lambda\in\C\colon |\lambda|\in |q|^{\Z}\big\}\cup \{0\}\), satisfying  
	the formal commutation relation~\(vnv^* = qn\)~\cite{RR2021}.
	
	We face similar kinds of technical limitations when trying to extend this duality for purely 
	complex (\(\textup{Im}(q)\neq 0\)) values of \(q\) as we did while building braided \(\qE\) groups and the reader is referred
	to the introduction in~\cite{RR2021} for details. 
  
	The required preliminaries are gathered in the next section. 
	The main content is contained in section~\ref{sec:Dual-Eq2}. In the last section, we construct 
	the bosonization, an analogous construction of the semidirect products of groups, of the dual 
	quantum group which yields an ordinary noncompact locally compact \(\Cst\)\nb-quantum group.

\section{Braided locally compact quantum groups - definition and duality}
  \label{sec:prelim} 
  	In this section we describe what braided locally compact quantum groups are 
	and what do we mean by their duals.
  	The generators of the \(\Cst\)\nb-algebras we shall encounter here are unbounded and hence 
	we shall use the notions of \(\Cst\)\nb-algebras generated by unbounded elements affiliated 
	with the \(\Cst\)\nb-algebra in the sense of Woronowicz~\cites{W1991b, W1995}. 
  
  	The notations used in~\cite{RR2021} are significantly relied upon in this text. As a result, 
	we try to keep the preliminaries  brief. Therefore, for any undefined notation we refer to 
	the aforementioned articles. Nevertheless, we reiterate some of the ideas here for the 
	readers' convenience.
	
	For a \(\Cst\)\nb-algebra \(A\), we denote by 
	\(\Mult(A)\), the multiplier algebra of \(A\) and by \(\U(A)\) the group of unitary multipliers 
	of~\(A\). For a Hilbert space~\(\Hils\), the identity operator on \(\Hils\) is denoted 
	by \(\Id_{\Hils}\) whereas the unit element of \(\Bound(\Hils)=\Mult(\Comp(\Hils))\) is denoted 
	by \(1_{\Hils}\). Also for two norm closed subsets \(X,Y\subset\Bound(\Hils)\)
	\[
	    X\cdot Y=\{xy\mid x\in X \text{ and } y\in Y\}^{\textup{CLS}},
	\]
	where ``CLS'' stands for the closed linear span. A morphism from \(A\) to \(B\) is a nondegenerate 
	\Star{}homomorphism \(\varphi\colon A\to \Mult(B)\) such that \(\varphi(A)B=B\). 
	The set of all morphisms from \(A\) to \(B\) is denoted by \(\Mor(A,B)\).
	Any quantum group \(\G\) is a pair
  	\((\Contvin(\G),\Comult[\G])\) consisting of a \(\Cst\)\nb-algebra 
	\(\Contvin(\G)\) and a morphism \(\Comult[\G]\) satisfying the coassociativity condition 
	\((\Id_{\Contvin(\G)}\otimes \Comult[\G])\circ \Comult[\G]{=}(\Comult[\G]\otimes\Id_{\Contvin(\G)})\circ \Comult[\G]\)
	 and the cancellation conditions
	 \(\Comult[\G](\Contvin(\G))(1\otimes \Contvin(\G))
	 {=}\Contvin(\G)\otimes \Contvin(\G)
	 {=}\Comult[\G](\Contvin(\G))(\Comult[\G](\Contvin(\G))\otimes 1)\).
	  For a complete definition of a \(\Cst\)\nb-quantum group see~\cites{MRW2014,SW2007}. 
	  We write~\(\G{=}(\Contvin(\G),\Comult[\G])\) to associate the underlying \(\Cst\)\nb-algebra 
	  and the comultiplication map.

\subsection*{Quantum group representations}

      	A \textup(right\textup) \emph{representation} of~\(\G\) on a \(\Cst\)\nb-algebra \(D\)
    	is an element \(\corep{U}\in\U(D\otimes \Contvin(\G))\) with
       	\begin{equation}
          \label{eq:corep_cond}
    	     (\Id_{D}\otimes\Comult[\G])\corep{U} =\corep{U}_{12}\corep{U}_{13}
    	     \qquad\text{in }\U(D\otimes \Contvin(\G)\otimes \Contvin(\G)).
         \end{equation} 
	In particular, if \(D=\Comp(\Hils[L])\) for some Hilbert space~\(\Hils[L]\), then \(\corep{U}\) is said to be 
	a \textup(right\textup) 
	representation of~\(\G\) on~\(\Hils[L]\).

   	The tensor product of two representations~\(\corep{U}^{i}\in\U(\Comp(\Hils[L]_{i})\otimes \Contvin(\G))\) of
    	\(\G\) on \(\Hils[L]_{i}\) for~\(i=1,2\) is a representation on \(\Hils[L]_1\otimes\Hils[L]_2\)
   	defined by
   	\begin{equation*}
       	    \corep{U}^{1}\tenscorep \corep{U}^{2}\defeq 
            \corep{U}^{1}_{13}\corep{U}^{2}_{23}
            \quad\text{ in }\U(\Comp(\Hils[L]_1\otimes\Hils[L]_2)\otimes \Contvin(\G)).
   	\end{equation*}
   
   	According to the notations above, if \((t\otimes 1_{\G})\corep{U}^1 =\corep{U}^2(t\otimes 1_{\G})\) 
	for some~\(t\in\Bound(\Hils[L]_1,\Hils[L]_2)\), then we say that \(t\) 
    	intertwines \(\corep{U}^{1}\) and \(\corep{U}^{2}\) and such an element is referred to as an intertwiner. 
	The set of all intertwiners between \(\corep{U}^{1}\) and \(\corep{U}^{2}\) is denoted as 
	\(\Hom^{\G}(\corep{U}^1, \corep{U}^2)\). A simple calculation demonstrates that the \(\tenscorep\) is 
	associative, and the trivial \(1\)\nb-dimensional representation is a tensor unit. 
	This establishes a \(\textup{W}^*\)\nb-category structure on the collection of representations, 
	which we denote by~\(\Corepcat(\G)\); for additional information, see~\cite[Section 3.1--2]{SW2007}. 
	Pairs \((\Hils[L],\corep{U})\) consisting of a Hilbert space \(\Hils[L]\) and a representation 
	of \(\G\) on \(\Hils[L]\) are objects of \(\Corepcat(\G)\).

\subsection*{Quantum group actions}
 
        A \textup(right\textup) action of~\(\G\) on a
        \(\Cst\)\nb-algebra~\(C\) is an injective morphism \(\Gamma\in\Mor(C,C\otimes\Contvin(\G))\) 
       with the following properties:
       \begin{enumerate}
         \item \(\Gamma\) is a comodule structure, that is,
             \begin{equation}
               \label{eq:right_action}
                 (\Id_C\otimes\Comult[\G])\circ\Gamma
                 = 
                 (\Gamma\otimes\Id_{\Contvin(\G)})\circ\Gamma;
             \end{equation}
  	\item \(\Gamma\) satisfies the \emph{Podle\'s condition}:
    	   \begin{equation}
      	      \label{eq:Podles_cond}
      		\Gamma(C)(1_C\otimes \Contvin(\G))=C\otimes \Contvin(\G).
    	   \end{equation}
       \end{enumerate}
       
       The pair~\((C,\Gamma)\) is referred to as a \emph{\(\G\)\nb-\(\Cst\)\nb-algebra}. 
       When it is obvious from the context, we shall drop~\(\Gamma\) from our notation and 
       merely write \(C\) is a \(\G\)\nb-\(\Cst\)\nb-algebra.
     
       	A morphism \(f\colon C\to D\) between two~\(\G\)\nb-\(\Cst\)\nb-algebras 
     	\((C,\Gamma_{C})\) and \((D,\Gamma_{D})\) is said to be 
	\emph{\(\G\)\nb-\hspace{0pt}equivariant} if 
     	\[
	   \Gamma_{D}\circ f =(f\otimes\Id_{\G})\circ\Gamma_{C}.
	\]
	
	The set of \(\G\)\nb-equivariant morphisms from \(C\) to \(D\) is denoted by \(\Mor^{\G}(C, D)\), 
	while the category containing \(\G\)\nb-\(\Cst\)\nb-algebras as objects and \(\G\)\nb-equivariant 
	morphisms as arrows is denoted by \(\Cstcat(\G)\). 
	
	For all~\(f\in\Contvin(G)\), \(g\in G\), every continuous action~\(\phi\) of a locally 
	compact group \(G\) on a \(\Cst\)\nb-algebra \(D\) produces an action \(\Gamma_{D}\) 
	of \(\G{=}(\Contvin(G),\Comult[\Contvin(G)])\) on \(D\) given 
	by~\((\Gamma_{D}(d) f)(g)\defeq f(\phi_{g}(d))\)  and vice versa. 
	Additionally, when \(G\)\nb-\(\Cst\)-algebras are used as objects and \(G\)\nb-equivariant 
	morphisms are used as arrows, the category \(\Cstcat(\G)\) is equivalent to the category \(\Cstcat(G)\). 
	Analogously, \(\Corepcat(\G)\) is equivalent to the representation category \(\Corepcat(G)\).

\subsection*{Quasitriangular quantum groups}
  \label{sec:Quasitriag}
   
	Let~\(\dual{\G}=(\Contvin(\DuG),\DuComult[\G])\) be the dual of \(\G\). 
     	An element \(\Rmat\in\U(\Contvin(\DuG)\otimes\Contvin(\DuG))\) is said to be 
	an~\(\Rmattxt\)\nb-matrix on~\(\DuG\) if it is a bicharacter, that is , it satisfies the following:
  	\[
    	   (\Id_{\Contvin(\dual{\G})}\otimes\DuComult[\G])\Rmat=\Rmat_{12}\Rmat_{13}, 
    	   \quad 
  	   (\DuComult[\G]\otimes\Id_{\Contvin(\dual{\G})})\Rmat =\Rmat_{23}\Rmat_{13},
 	 \]
	 in \(\U(\Contvin(\dual{\G})\otimes\Contvin(\dual{\G})\otimes\Contvin(\dual{\G}))\)
  	and satisfies the~\(\Rmattxt\)\nb-matrix condition:
    	\begin{equation}
  	  \label{eq:R-mat}
  	    \Rmat(\flip\DuComult[\G](\hat{a}))\Rmat^{*}=\DuComult[\G](\hat{a}) 
  	    \qquad\text{for all~\(\hat{a}\in\Contvin(\dual{\G})\).}
 	\end{equation}
	A \emph{quasitriangular quantum group} is a quantum group~\(\G\) with an \(\Rmattxt\)\nb-matrix
	 \(\Rmat\in\U(\Contvin(\dual{\G})\otimes\Contvin(\dual{\G}))\)~\cite[Section 3]{MRW2016}. 
 
    	Recall from~\cite[Lemma 2.2]{MRW2016} that
      	the dual \(\hat{\Rmat}\defeq \flip(\Rmat^*)\) of a bicharacter 
	\(\Rmat\in\U(\Contvin(\dual{\G})\otimes \Contvin(\dual{\G}))\) is an
      	\(\Rmattxt\)\nb-matrix if and only if \(\Rmat\) is an \(\Rmattxt\)\nb-matrix. 
	   
   	The categories \(\Corepcat(\G)\) and~\(\Cstcat(\G)\) are particularly interesting when \(\G\) 
	is quasitriangular. More specifically, the representation category \(\Corepcat(\G)\) is a braided 
	monoidal category and \(\Cstcat(\G)\) is a monoidal category by virtue 
	of~\cite[Proposition 3.2 \& Theorem 4.3]{MRW2016}. The braiding isomorphism in~\(\Corepcat(\G)\) 
	is a unitary operator. Such a category is referred to as a unitarily braided monoidal category. 
	We take into consideration the monoidal product~\(\boxtimes_{\DuRmat}\) on~\(\Cstcat(\G)\) and 
	the dual braiding corresponding to the unitary braiding~\(\Braiding{}{}\) on~\(\Corepcat(\G)\) whose 
	explicit construction is described in~\cites{RR2021,MRW2014, MRW2016}. 
	The latter construction was motivated by~\cite{NV2010}.  We briefly recall the construction 
	for \(\G{=}\T\) in the following subsection.

\subsection*{Duality of braided quantum groups}
	The circle group~\(\T\) can be viewed as a quasitriangular quantum group.   
   	Every bicharacter~\(\chi\colon \Z\times \Z\to\T\) satisfies the~\(\Rmattxt\)\nb-matrix 
  	condition~\eqref{eq:R-mat} because \(\Contvin{(\dual{\T})}=\Contvin(\Z)\) is a commutative 
  	\(\Cst\)\nb-algebra. Thus~\(\T\) is quasitriangular with respect to any bicharacter on~\(\Z\).
	In particular, for~\(q\in\C\setminus\{0\}\) the bicharacter~\(\Rmat\colon\Z\times\Z\to\T\) 
  	defined by \(\Rmat(m,n)=\zeta^{mn}\) where \(\zeta=\frac{q}{\conj{q}}\)  
	is an \(\Rmattxt\)\nb-matrix on~\(\Z\).
   
	
	Let~\(\corep{U}\in\U(\Comp(\Hils[L])\otimes \Cont(\T))\) be a representation of \(\T\) on \(\Hils[L]\). 
	Then there is a unique solution 
   	 \(Z\in\U(\Hils[L]\otimes\Hils[L])\) such that
    	\begin{equation}
      	  \label{eq:Z-expr}
             \corep{U}_{1\alpha}\corep{U}_{2\beta}Z_{12}
      	     =
      	     \corep{U}_{2\beta}\corep{U}_{1\alpha}
      	     \quad
      	     \textup{ in }\U(\Hils[L]\otimes \Hils[L]\otimes \Hils[K])
    	\end{equation}
    	for any \(\Rmat\)\nb-Heisenberg pair \((\alpha,\beta)\) on any Hilbert space \(\Hils[K]\), 
	see~\cite[Theorem 4.1]{MRW2014}. 
     	Hence, the unitary \(Z\) determines the braiding~\(\Braiding{}{}\defeq Z\circ\Flip\). It turns out that the 
	dual braiding~\(\Dualbraiding{}{}\defeq\widehat{Z}\circ\Flip\) where \(\dual{Z}=\Flip Z^* \Flip\), governed by the dual 
	\(\Rmat\)\nb-matrix~\(\dual{\Rmat}\), is inverse to the original braiding. 
	Using the fact that a pair of representations \((\alpha,\beta)\) is an \(\Rmat\)\nb-Heisenberg pair if and only if 
	\((\beta,\alpha)\) is an \(\DuRmat\)\nb-Heisenberg pair, it follows that
	\(\dual{Z}\) satisfies a variant of equation~\eqref{eq:Z-expr}: 
	\begin{equation}
      	  \label{eq:Zhat-expr}
             \corep{U}_{1\beta}\corep{U}_{2\alpha}\widehat{Z}_{12}
      	     =
      	     \corep{U}_{2\alpha}\corep{U}_{1\beta}
      	     \quad
      	     \textup{ in }\U(\Hils[L]\otimes \Hils[L]\otimes \Hils[K]).
    	\end{equation}
    	For more information on Heisenberg pairs see~\cite{MRW2014}, and for the braiding operator see~\cite{MRW2016}.

%
    
     	Now suppose~\((C_{i},\Gamma_{i})\) are two objects in~\(\Cstcat(\T)\) for \(i=1,2\) 
	and let~\(\Dualbraiding{}{}\) be the unitary dual braiding. Let~\(C_i\hookrightarrow\Bound(\Hils[L]_i)\) 
	be the \(\T\)\nb-equivariant representations of \(C_i\) on \(\Hils[L]_i\) for \(i=1,2\) respectively. 
	Then the monoidal product \(\boxtimes_{\DuRmat}\) on \(\Cstcat(\T)\) is given by
    	\[
             C_{1}\boxtimes_{\DuRmat} C_{2}
             \defeq 
             \left\{\dual{j_1}(x)\dual{j_2}(y)\mid x\in C_1, y\in C_2\right\}^{\textup{CLS}}
             \subset
             \Bound(\Hils[L]_1\otimes\Hils[L]_2),
    	\]
    	where \(\dual{j_1}(x)=x\otimes \Id_{\Hils[L]_1}\) and 
	\(\dual{j_2}(y)=\Dualbraiding{}{}(y\otimes\Id_{\Hils[L]_2})\Dualbraiding{}{}^*\) 
	for all~\(x\in C_1\) and \(y\in C_2\).
  	
	Then~\(C_{1}\boxtimes_{\DuRmat} C_{2}{=} \dual{j_{1}}(C_1)\dual{j_{2}}(C_2)\) 
    	is a~\(\Cst\)\nb-algebra~\cite[Theorem 4.6]{MRW2014} and
    	\[
      	   \Gamma_{1}\bowtie\Gamma_{2}(\dual{j_{1}}(c_1)\dual{j_{2}}(c_2))
      	   \defeq 
      	   (\dual{j_{1}}\otimes\Id_{A})\Gamma_{1}(c_1)(\dual{j_{2}}\otimes\Id_{A})\Gamma_{2}(c_2)
    	\] 
    	defines the diagonal action of~\(\T\) on~\(C_1\boxtimes_{\DuRmat}C_2\)~\cite[Proposition 4.1]{MRW2016}.
    	Hence, \(C_1\boxtimes_{\DuRmat}C_2\) is an object in \(\Cstcat(\T)\). 
	For two objects~\((C'_{i},\Gamma'_{i})\)  in~\(\Cstcat(\T)\) 
	and~\(f_{i}\in\Mor^{\T}(C_{i},C'_{i})\) for~\(i=1,2\), the monoidal product 
	\(f_{1}\boxtimes_{\DuRmat}f_{2}\) is defined in a canonical way.

	The multiplicative unitary, which simultaneously encodes much of the information of a quantum 
	group and it's dual, is one of the key elements in the theory of quantum groups in the operator 
	algebraic framework~\cites{BS1993,W1996,SW2007}.  Along with a full investigation of a more 
	generic entity known as the braided multiplicative unitary, the theory of braided \(\Cst\)\nb-quantum 
	groups is explored in~\cite{MRW2017}. More explicitly, for a Hilbert space \(\Hils[L]\), 
	a unitary operator~\(\BrMultunit\) on \(\Hils[L]\otimes \Hils[L]\) is said to be a multiplicative unitary 
	if it satisfies the pentagonal equation
	\[
   	    \BrMultunit_{23}\BrMultunit_{12}=\BrMultunit_{12}\BrMultunit_{13}\BrMultunit_{23} 
   	    \quad 
  	    \text{ in }\U(\Hils[L]\otimes\Hils[L]\otimes\Hils[L]).
	\]
	More assumptions such as modularity and manageability are required on \(\BrMultunit\) 
	to construct \(\Cst\)\nb-quantum groups. Braided analogues of a multiplicative 
	unitary and its manageability are defined in~\cite[Definition 3.2]{MRW2017} 
	and~\cite[Definition 3.5]{MRW2017} respectively.
	We will, however, use an alternative but equivalent approach to the one established 
	in~\cite{MRW2017}. As a result, our definition of the notion of 
	braided multiplicative unitary over a regular quasitriangular quantum group becomes 
	more concise one, see~\cite[Definition3.1 and Definition 3.4]{RR2021}.
    
    	\begin{theorem}
      	  \label{the:manag-DuBrMultunit}
  		Let~\(\BrMultunit\in\U(\Hils[L]\otimes\Hils[L])\) be a manageable braided multiplicative 
		unitary over \(\T\) with respect to \((\corep{U},\Rmat)\). 
  		Then the dual \(\DuBrMultunit\defeq\Dualbraiding{}{}\BrMultunit^{*}\Dualbraiding{}{}^*\) 
		is a manageable braided multiplicative unitary over \(\T\) with respect to \((\corep{U},\DuRmat)\). 
    	\end{theorem}

    	\begin{proof}
    		The proof essentially follows the arguments used in the proof of~\cite[Proposition 3.6]{MRW2017}.
  		Let~\(\widetilde{\BrMultunit}\) and \(Q\) be the operators that witness the manageability 
		of the braided multiplicative unitary \(\BrMultunit\). And let \(Z\) and \(\widetilde{Z}\) 
		be the operators appearing in~\cite[Equation 3.5]{RR2021}.
	
		The \(\Rmattxt\)\nb-matrix \(\Rmat\) being a bicharacter, it induces a quantum group 
		homomorphism from \(\T\) to \(\Z\)~\cite{MRW2012} and hence induces a representation 
		\(\corep{V}\in\U(\Comp(\Hils[L])\otimes \Contvin(\Z))\) satisfying~\cite[Equation (32)]{MRW2012}. 
		Recall that \(\Corep{W}\) is a multiplicative unitary generating \(\T\). 
		Let~\(\pi\in\Mor(\Cont(\T),\Comp(\Hils))\) and \(\hat{\pi}\in\Mor(\Contvin(\Z),\Comp(\Hils))\)
		be the embeddings as in~\cite[Theorem 2]{SW2007}, see also~\cites{W1996,SW2001}.
        		Now~\(Q_{\Hils[L]}\otimes Q\) commutes with~\(\Corep{U}=(\Id_{\Hils[L]}\otimes\pi)\corep{U}\). Also, 
		we have observed in the proof of~\cite[Theorem 3.6]{RR2021} that~\(Q_{\Hils[L]}\otimes Q\) 
		commutes with~\(\Corep{V}=(\Id_{\Hils[L]}\otimes\hat{\pi})\corep{V}\). 
		Moreover, let \((\Hils[L],\corep{U}^1)\) and \((\Hils[L],\corep{U}^2)\) be elements 
		in~\(\Corepcat(\T)\) and \((\Hils[L],\corep{V}^1)\) and \((\Hils[L],\corep{V}^2)\) be the 
		corresponding elements in~\(\Corepcat(\dual{\T})\).
	Then, \(Z\) satisfies
	\[
	  \left( (\Id_{\Hils[L]}\otimes \pi)\corep{U}^1\right)_{13}
	  \left( (\Id_{\Hils[L]}\otimes \hat{\pi})\corep{V}^2\right)_{23}
	  Z_{12}
	  =
	  \left( (\Id_{\Hils[L]}\otimes \hat{\pi})\corep{V}^2\right)_{23}
	  \left( (\Id_{\Hils[L]}\otimes \pi)\corep{U}^1\right)_{13}
	\]
	in~\(\U(\Hils[L]\otimes\Hils[L]\otimes\Hils)\);
	hence it commutes with \(Q_{\Hils[L]}\otimes Q_{\Hils[L]}\).

	Following a similar line of arguments as in~\cite[Proposition 1.4(2)]{W1996} we obtain
	\begin{equation}
	 \label{eq:manageable-aux}
	 \langle x\otimes u \mid Z^*\BrMultunit \mid y\otimes v\rangle
	 =
	 \langle \conj{Q_{\Hils[L]}(y)}\otimes u\mid\widetilde{\BrMultunit}\widetilde{Z}\mid
	 \conj{Q_{\Hils[L]}^{-1}(x)}\otimes v\rangle
	\end{equation}
	for \(y\in\dom(Q_{\Hils[L]})\), \(x\in\dom(Q_{\Hils[L]}^{-1})\), \(u,v\in\Hils[L]\).
	
	Let~\(\widetilde{\widehat{Z}},\widetilde{\DuBrMultunit}\in\U(\conj{\Hils[L]}\otimes\Hils[L])\) be the 
	unitaries defined by
	\[
	  \widetilde{\widehat{Z}}\defeq (\Flip\widetilde{Z}^*\Flip)^{\transpose\otimes\transpose},
	  \qquad
	  \widetilde{\DuBrMultunit}\defeq (\Flip\widetilde{\BrMultunit}^*\Flip)^{\transpose\otimes\transpose}.
	\]
	Then by definition, \(\widehat{Z}^*\DuBrMultunit=\Flip\BrMultunit^* Z\Flip\) and 
	\(\widetilde{\DuBrMultunit}\widetilde{\widehat{Z}}=(\Flip\widetilde{Z}^*\widetilde{\BrMultunit}^*\Flip)^{\transpose\otimes\transpose}\).
	Verification of the manageability condition is similar to the computations used 
	in~\cite[Proposition 3.6]{MRW2017} and therefore we omit it.

	Since for the dual braided multiplicative unitary the unitary \(Z\) becomes \(\dual{Z}\), the operator
	\(\widetilde{\DuBrMultunit}\) witnesses the manageability of \(\DuBrMultunit\).
    \end{proof}

	The manageability of the dual braided multiplicative unitary \(\DuBrMultunit\) enables 
	us to construct a braided quantum group from it in the following sense.

    	\begin{theorem}
     	  \label{the:DuBQgrp-quasitriag}
    		Let~\(\BrMultunit\) be a manageable braided multiplicative unitary over \(\T\) 
		relative to \((\corep{U},\Rmat)\) and \(\DuBrMultunit=\Dualbraiding{}{}\BrMultunit^{*}\Dualbraiding{}{}^*\). 
		Let 
		\begin{equation}
   	    	    \Contvin(\DuG[H])\defeq
	    	    \{(\omega\otimes\Id_{\Comp(\Hils[L])})\DuBrMultunit\mid\omega\in\Bound(\Hils[L])_{*}\}^{\textup{CLS}}
	    	    \subset\Bound(\Hils[L]).
 		\end{equation}
 		Then
 		\begin{enumerate}
  	   	    \item \(\Contvin(\DuG[H])\) is a nondegenerate, separable \(\Cst\)\nb-subalgebra 
		    		of~\(\Bound(\Hils[L])\)\textup{;}
  	   	    \item Define~\(\dual{\Gamma}(\hat{b})\defeq \corep{U}(\hat{b}\otimes 1_{A})\corep{U}^{*}\) 
		    		for all~\(\hat{b}\in \Contvin(\DuG[H])\). 
	   			Then \(\dual{\Gamma}\in\Mor(\Contvin(\DuG[H]),\Contvin(\DuG[H])\otimes \Cont(\T))\) 
				and~\((\Contvin(\DuG[H]),\dual{\Gamma})\) is an object of~\(\Cstcat(\T)\)\textup{;}
  	   	    \item \(\DuBrMultunit\in\U(\Comp(\Hils[L])\otimes \Contvin(\DuG[H]))\)\textup{;}
 	   	    \suspend{enumerate}
	     	   	Consider the twisted tensor product~\(\Contvin(\DuG[H])\boxtimes_{\DuRmat}\Contvin(\DuG[H])\). 
  	     	   	Suppose \(\dual{j_{1}},\dual{j_{2}}\) are the canonical embeddings : 
		   	\(\dual{j_{1}},\dual{j_{2}}\in\Mor^{\T}(\Contvin(\DuG[H]),\Contvin(\DuG[H])\boxtimes_{\DuRmat}\Contvin(\DuG[H]))\). 
	  	   \resume{enumerate}  
 	   	   \item There exists a 
		   		unique~\(\Comult[{\DuG[H]}]\in
				\Mor^{\T}(\Contvin(\DuG[H]),\Contvin(\DuG[H])\boxtimes_{\DuRmat}\Contvin(\DuG[H]))\) 
				characterised by 
  			\begin{equation}
   		  	  \label{eq:ComulB-R}
      		  	   (\Id_{\Comp(\Hils[L])}\otimes\Comult[{\DuG[H]}])\DuBrMultunit
		  	   =
		  	   \bigl((\Id_{\Comp(\Hils[L])}\otimes \dual{j_{1}})\DuBrMultunit\bigr) 
      		  	   \bigl((\Id_{\Comp(\Hils[L])}\otimes \dual{j_{2}})\DuBrMultunit\bigr).
   			\end{equation}
   			Moreover, \(\Comult[{\DuG[H]}]\) is coassociative: 
			\((\Comult[{\DuG[H]}]\boxtimes_{\DuRmat}\Id_{\Contvin(\DuG[H])})\circ\Comult[{\DuG[H]}]
			=(\Id_{\Contvin(\DuG[H])}\boxtimes_{\DuRmat}\Comult[{\DuG[H]}])\circ\Comult[{\DuG[H]}]\)
   			and satisfy the cancellation conditions: 
			\[
			\Comult[{\DuG[H]}](\Contvin(\DuG[H]))\dual{j_{1}}(\Contvin(\DuG[H]))
			=\Contvin(\DuG[H])\boxtimes_{\DuRmat}\Contvin(\DuG[H])
			=\Comult[{\DuG[H]}](\Contvin(\DuG[H]))\dual{j_{2}}(\Contvin(\DuG[H])).
			\]
 		\end{enumerate} 
 		The triple~\((\Contvin(\DuG[H]),\Comult[{\DuG[H]}],\dual{\Gamma})\) is said to be the 
		\emph{braided \(\Cst\)\nb-quantum group} \textup{(}over~\(\T\)\textup{)} 
		generated by the braided multiplicative unitary~\(\DuBrMultunit\) relative to \((\corep{U},\DuRmat)\).
    	\end{theorem}
    
    \begin{proof}
      See~\cite[Theorem 5.1]{Roy2022}.
    \end{proof}

 	\begin{remark}
     		Theorem~\ref{the:manag-DuBrMultunit} and Theorem~\ref{the:DuBQgrp-quasitriag} 
		remain valid when \(\T\) is replaced by a regular quasitriangular quantum group~\(\G\), 
		see for example~\cites{RR2021,Roy2022}.
 	\end{remark}
 
\section{Braided duallity for  quantum \(\textup{E}(2)\) groups}  
   \label{sec:Dual-Eq2}
   
   	The construction of the braided dual of \(\qE\) groups in the sense of Theorem~\ref{the:DuBQgrp-quasitriag} 
	will be the focus of this section. We begin by constructing the dual braided multiplicative unitary.
   
   	Let~\(q\in\C\setminus \{0\}\) such that \(\Mod{q}<1\). Through the rest of the article we shall make implicit use 
	of this assumption without explicitly mentioning it.
	     
\subsection{The multiplicative unitary \(\DuBrMultunit\)}
      \label{sec:Du-Brmultunit-E2}
      	Let~\(\Hils=\ell^2(\Z)\) with an orthonormal basis \(\{e_{i}\}_{i\in\Z}\) and identify \(\Hils[L]=\Hils\otimes \Hils\). 
	Then, \(\{e_{i,j}=e_{i}\otimes e_j\}_{i,j\in\Z}\) is an orthonormal basis for \(\Hils[L]\).
	
	A pair of operators~\((v,n)\) such that 
	\begin{equation}
	 \label{eqn:Eq2-gen}
	v \text{ is unitary }, \quad 
	n \text{ is normal }, \quad
	vnv^*=qn, \quad \text{ and }\quad
	\Sp(n)=\Specn, 
	\end{equation}
	generates the quantum \(\textup{E}(2)\) group in the sense of~\cite{W1995,W1991b}. Equation~\eqref{eqn:Eq2-gen} can be realized 
	concretely as operators on \(\Hils[L]\) given by 
	\begin{equation}
	  \label{eq:v-n-rep}
	v e_{ij}=e_{i-1,j} \quad \text{ and }\quad  n e_{ij}=q^i e_{i,j+1}, 
	\end{equation}
	see~\cite[Equation (1.6)]{RR2021}. In fact, any Hilbert space realization of \((v,n)\) is either 1-dimensional or infinite dimensional, 
	and the direct integral of all infinite dimensional irreducible representations of \((v,n)\) is unitarily equivalent to~\eqref{eq:v-n-rep}, 
	see~\cite{W1991b}*{Sec. 3(B)}. Therefore, \eqref{eq:v-n-rep} defines a faithful 
	nondegenerate \Star{}representation of \(\Contvin(\qE)\) on \(\Hils[L]\).
	
      	The braided multiplicative unitary \(\BrMultunit\) generating the braided quantum \(\textup{E}(2)\) groups is described in terms of \(v,n\) and an additional unitary \(P\in\U(\Hils[L])\) defined by \(P e_{i,j}=\zeta^{-j}e_{i,j}\). In order to construct the braided multiplicative unitary corresponding to the dual quantum group, we first define
       the operators \(N\) and \(b\)  on \(\Hils[L]\) by
    	\begin{equation}
      	   \label{def:N_and_b}
      	       N e_{i,j}=(i+j)e_{i,j}, \qquad b e_{i,j}=q^{\frac{j-i}{2}} e_{i-1,j-1}.
    	\end{equation}
	Then \(N\) is selfadjoint with spectrum \(\Z\), \(b\) is normal and injective, implying that it is invertible.
    	Moreover, using the action of action of \(n^{-1}\) on the basis elements given by \(n^{-1}e_{i,j}=q^{-i}e_{i,j-1}\),
	 one can easily check that  
	\begin{equation}
	  \label{eq:vn-rel-Nb}
	   n^{-1}v=q^{-\frac{N}{2}}b.
	\end{equation}
    
    	Referring the operator \(N + 2I\), where \(I\) is the identity operator on \(\Hils[L]\), 
	as \(\tilde{N}\) throughout this article, we define the operators \(\hat{X}\) and \(\DuCorep{Y}\) 
    	on \(\Hils[L]\otimes \Hils[L]\) by
    	\begin{align*}
       	    \hat{X} &\defeq Pb^{-1}q^{\frac{\tilde{N}}{2}}\otimes q^{\frac{\tilde{N}}{2}}b,\\
            \DuCorep{Y}&\defeq (v\otimes 1)^{1\otimes N}.
    	\end{align*} 
        Then \(\DuCorep{Y}\) is a unitary and \(\hat{X}\) is a closed operator with spectrum 
        \(\Sp(\hat{X})\) contained in \(\Specn\). Additionally, on the basis elements the actions are 
        given by
    	\(\hat{X}e_{i,j}\otimes e_{k,l}=\zeta^{-j-1}q^{i+l+1} e_{i+1,j+1}\otimes e_{k-1,l-1}\)
    	and
    	\(\DuCorep{Y}e_{i,j}\otimes e_{k,l}=e_{i-k-l,j}\otimes e_{k,l}\).
	Recall the manageable braided multiplicative unitary \(\BrMultunit\) generating 
	the braided quantum \(\textup{E}(2)\) groups from~\cite[Theorem 4.6]{RR2021}. Then
	with these operators at our disposal
	we construct the corresponding dual braided multiplicative unitary in the next theorem.

	Also recall the quantum exponential function \(\QuantExp\colon \Specn\to\T\) defined 
	in~\cite[Equation 1.2]{W1992a} by
	\[
  	  \QuantExp(\lambda)=
     	    \begin{cases}
     	      \prod_{k=0}^{\infty}\frac{1+\Mod{q}^{2k}\conj{\lambda}}{1+{\Mod{q}}^{2k}\lambda}, 
	      		&\text{ if } k\in \Specn\setminus\{-\Mod{q}^{-2k}\mid k=0,1,2,\ldots\}\\
     		-1, 	&\text{ otherwise.}
     	   \end{cases}
	\]
	Since \(\QuantExp\) is a unitary multiplier of \(\Contvin(\Specn)\) we observe that 
	\(\QuantExp(\hat{X})\in\Bound(\Hils[L]\otimes \Hils[L])\).

   	\begin{theorem}
     	   \label{the:DuMultunit-main}
   		The operator~\(\DuBrMultunit\defeq\QuantExp(\hat{X})^*\DuCorep{Y}\) is a 
   		manageable braided multiplicative unitary on \(\Hils[L]\otimes \Hils[L]\) over \(\T\) 
		with respect to \((\corep{U},\DuRmat)\).
   	\end{theorem}
 
 	\begin{proof}
 		The dual braiding induced by the dual  \(\Rmattxt\)\nb-matrix \(\DuRmat\) is given by
 		\(\Dualbraiding{}{}=\hat{Z}\circ \Flip\). 
 		Equation~\eqref{eq:Zhat-expr} and a computation analogous to~\cite[Equation 4.2]{RR2021} 
		indicate that \(\hat{Z}\) operates on \(\Hils[L]\otimes \Hils[L]\) by
 		\[
   		   \hat{Z} e_{ij}\otimes e_{kl}=\zeta^{jl} e_{ij}\otimes e_{kl}.
 		\]
 		A simple computation yields
 		\(\hat{Z}(vn\otimes n^{-1}vP)\hat{Z}^*=Pvn\otimes n^{-1}v\). Indeed, 
 		\begin{align*}
   		    &\hat{Z}(vn\otimes n^{-1}vP)\hat{Z}^* e_{i,j}\otimes e_{k,l} \\
   		    &\quad=\zeta^{-jl}\hat{Z}(vn\otimes n^{-1}vP)e_{i,j}\otimes e_{k,l}\\
   		    &\quad=\zeta^{-jl-l}\hat{Z}q^{i}e_{i-1,j+1}\otimes q^{-k+1}e_{k-1,l-1}\\
   		    &\quad=\zeta^{-jl-l+(j+1)(l-1)}q^{i}e_{i-1,j+1}\otimes q^{-k+1}e_{k-1,l-1}\\
  		    &\quad=\zeta^{-j-1}q^{i}e_{i-1,j+1}\otimes q^{-k+1}e_{k-1,l-1}\\
   		    &\quad=(Pvn\otimes n^{-1}v)e_{ij}\otimes e_{kl}
 		\end{align*}
 		Since by definition \(\DuCorep{Y}\) shifts only the first factor of the basis vectors, 
 		\(\hat{Z}\) commutes with \(\DuCorep{Y}\). 
 		Now,
		\begin{align*}
   		    \DuCorep{Y}(Pvn\otimes n^{-1}v)\DuCorep{Y}^* e_{i,j}\otimes e_{k,l}
       		   &=\DuCorep{Y}(Pvn\otimes n^{-1}v) e_{i+k+l,j}\otimes e_{k,l} \\
       	 	   &=\DuCorep{Y} \zeta^{-j-1}q^{i+k+l} e_{i+k+l-1,j+1}\otimes q^{-k+1}e_{k-1,l-1} \\
       		   &= \zeta^{-j-1}q^{i} e_{i+1,j+1}\otimes q^{l+1}e_{k-1,l-1}\\
       		   &= (Pv^{*}n\otimes q^2 q^{\frac{N}{2}}b) e_{i,j}\otimes e_{k,l}.
 		\end{align*}
 		From~\cite[Theorem 4.6]{RR2021}, we have~\(\BrMultunit{=}\QuantExp(n^{-1}vP \otimes vn)\Corep{Y}\).
		Here, \(\Corep{Y}\) is the (multiplicative) unitary on \(\Hils[L]\otimes \Hils[L]\) 
		defined by \( \Corep{Y} e_{ij}\otimes e_{kl}= e_{ij}\otimes e_{k+i+j,l}\).
 		Consequently, by combining the previous computations with the definition of~\(\DuBrMultunit\), 
		we have
 		\begin{align*}
   		   \DuBrMultunit 
		     &= \Dualbraiding{}{} \BrMultunit^{*} \Braiding{}{}=\hat{Z} \Flip\BrMultunit^*\Flip \hat{Z}^*\\
                      &=\hat{Z} \DuCorep{Y}\QuantExp(vn \otimes n^{-1}vP)^* \hat{Z}^*\\
                      &=\DuCorep{Y}\QuantExp(Pvn\otimes q^{-\frac{N}{2}}b)^{*}\\
                      &=\QuantExp(Pv^*n\otimes q^2 q^{\frac{N}{2}}b)^{*}\DuCorep{Y}.
 		\end{align*}
		Using equation~\eqref{eq:vn-rel-Nb} we get the desired result.
 
		By employing a technique similar to that described in~\cite[Proposition 4.8]{RR2021} 
		for \(\lambda = 1\), we see that \(\DuBrMultunit\) satisfies the braided pentagon 
		equation given below:
 		\begin{equation}
  		    \label{eq:DuBraided-pentagon}
 		      \DuBrMultunit_{23} \DuBrMultunit_{12} 
		      =\DuBrMultunit_{12}\Dualbraiding{}{}_{23} \DuBrMultunit_{12} \Dualbraiding{}{}_{23}^*  \DuBrMultunit_{23}.
		\end{equation}
		Since \(\BrMultunit\) is manageable by~\cite[Theorem 4.6]{RR2021}, 
		\(\DuBrMultunit\) is manageable by Theorem~\ref{the:manag-DuBrMultunit}.
 	\end{proof}
 
\subsection{The \(\Cst\)-algebra \(\Contvin(\dual{\qE})\)}
   \label{sec:Du-Eq2-cstalg}

 	Let~\(C\) be the \(\Cst\)\nb-algebra generated by an element \(\xi\) subject to the commutation relation 
 	\(\xi^* \xi=\Mod{q}^{-2}\xi \xi^*\) and \(\Sp(\xi) \subseteq \Specn\). Then by defintion, \(C\) is the closed linear 
	span of the set of all finite linear combinations of the form
 	\(\sum_{k}\Ph{\xi}^k g_{k}(\Mod{\xi})\) where \(\xi=\Ph{\xi}\Mod{\xi}\) is the polar decomposition 
	of \(\xi\), \(k\) runs over a finite subset of \(\Z\) and 
 	\(g_{k}\in\Contvin(\Sp(\Mod{\xi}))\) for all \(k\neq 0\) and \(g_0=0\).
 	Define \(\alpha\colon \T\to \Aut(C)\) by \(\alpha_{z}(g)(t)=g(zt)\) for \(g\in\Contvin(\Sp(\Mod{\xi}))\) and \(\alpha_{z}\Ph{\xi}=z\Ph{\xi}\).
 	Then the \(\Cst\)\nb-algebra \(\Contvin(\dual{\qE})\) of continuous functions vanishing at infinity 
	on \(\widehat{\qE}\) is isomorphic to the crossed product \(\Cst\)\nb-algebra \(C\rtimes \T\).

 	Let \(\tilde{b}\) denote the operator \(q^{\frac{\tilde{N}}{2}}b\).
	Let~ \(\tilde{b}=\Ph{\tilde{b}}\Mod{\tilde{b}}\) and \(q=\Ph{q}\Mod{q}\) be the polar decompositions of \(\tilde{b}\) and \(q\)
	respectively. 
	Then the actions of 
 	\(\Ph{\tilde{b}}\) and \(\Mod{\tilde{b}}\) on \(\Hils[L]\) are given by 
	\(\Ph{\tilde{b}}e_{ij}=\Ph{q}^{j-1} e_{i-1,j-1}\) and 
 	\(\Mod{\tilde{b}} e_{ij}=\Mod{q}^{j-1}e_{ij}\) respectively. Hence, it follows that 
 	\(\Ph{\tilde{b}}\Mod{\tilde{b}}\Ph{\tilde{b}}^*=\Mod{q}\Mod{\tilde{b}}\). 
	Consequently, we have \(\tilde{b}^*\tilde{b}=\Mod{q}^{-2}\tilde{b}\tilde{b}^*\).
 	Note that the spectrum of \(\tilde{b}\) is contained in \(\Specn\).
	By the universal property of the \(\Cst\)\nb-algebra generated by a finite set of elements, 
	there exists a \Star{}homomorphism
 	\[
   	  \sum_{k} \Ph{\xi}^{k} g_{k}(\Mod{\xi})\mapsto\sum_{k} \Ph{\tilde{b}}^{k} g_{k}(\Mod{\tilde{b}}).  
 	\]
	Clearly, \(\xi\mapsto \tilde{b}\) gives a representation of \(C\) on \(\Hils[L]\). Then \(\tilde{b}\) generates the \(\Cst\)\nb-algebra in the sense of~\cite{W1995}. 
  	Therefore, \(\tilde{b}\) is affiliated with \(C\).
	 Since, \(\T\) is commutative, \(C\rtimes \T\) contains
  	\(\Contvin(\Z)\) as a subalgebra and by definition \(N\Aff \Contvin(\Z)\).
	
	One can see that \(N\) and \(\tilde{b}\) separate representations of \(\Contvin(\dual{\qE})\) and that 
	\((1+N^2)^{-1}(1+\tilde{b}^*\tilde{b})^{-1}\) is in \(\Contvin(\dual{\qE})\). Thus we have the folowing
	
 	\begin{proposition}
 	  \label{prop:Du-qE-Cstalg}
 		The \(\Cst\)\nb-algebra \(\Contvin(\dual{\qE})\) is generated by \(N\) and \(\tilde{b}\).
 	\end{proposition}
  
 	Let~\(\dual{B} \defeq\left\{(\omega\otimes\Id)\DuBrMultunit \mid \omega\in\Bound(\Hils[L])_{*}\right\}^{\textup{CLS}}\).
	The general theory of manageable multiplicative unitaries ensures that \(\dual{B}\) is a \(\Cst\)\nb-algebra.
 	It is already observed that \(v^*n=b^{-1}q^{\frac{N}{2}}\). Now, \(n\) is a closed unbounded operator, implies \(n\Aff \Comp(\Hils[L])\). Also, since \(v\) and \(P\) are unitaries, 
	we obtain \(Pv^*n\Aff \Comp(\Hils[L])\). Moreover, \(\tilde{b}\) generates \(\Contvin(\dual{\qE})\) and therefore 
	\(Pb^{-1}q^{\frac{N}{2}}\otimes q^2 q^{\frac{N}{2}}b  \Aff \Comp(\Hils[L])\otimes \Contvin(\dual{\qE})\). 
	Consequently~\(\QuantExp(Pb^{-1}q^{\frac{\tilde{N}}{2}}\otimes q^{\frac{\tilde{N}}{2}}b)
	\in\U(\Comp(\Hils[L])\otimes \Contvin(\dual{\qE}))\). 
	Also observe that \(\Corep{Y}\) is a manageable multiplicative unitary that generates \(\T\) 
	as a quantum group. Hence, its dual \(\DuCorep{Y}\in\U(\Hils[L]\otimes\Hils[L])\) is also a 
	manageable multiplicative unitary and generates the dual \(\Z\) of \(\T\) as a quantum group. 
	Therefore, 
	\(\DuCorep{Y}\in\U(\Comp(\Hils[L])\otimes \Contvin(\Z))\subset \U(\Comp(\Hils[L])\otimes \Contvin(\dual{\qE}))\). 
	As a consequence, we have
	\(\DuBrMultunit\in\U(\Comp(\Hils[L])\otimes \Contvin(\dual{\qE}))\) and therefore, by definition of \(\dual{B}\), we have 
	\(\dual{B}\subseteq\Mult(\Contvin(\dual{\qE}))\). Moreover, 
	\(\DuBrMultunit(\Comp(\Hils[L])\otimes \Contvin(\dual{\qE}))=\Comp(\Hils[L])\otimes \Contvin(\dual{\qE})\) implies
	\begin{align*}
 	  \dual{B}\Contvin(\dual{\qE}) 
	      	&=\{(\omega\otimes\Id)\DuBrMultunit (1_{\Comp(\Hils[L])}\otimes \hat{b})
 			\mid \omega\in\Bound(\Hils[L])_{*}, \hat{b}\in \Contvin(\dual{\qE})\}^{\textup{CLS}} \\
       		&=\{(\omega\otimes\Id)\DuBrMultunit (m\otimes \hat{b})
			\mid m\in\Comp(\Hils[L]), \omega\in\Bound(\Hils[L])_{*}, \hat{b}\in \Contvin(\dual{\qE})\}^{\textup{CLS}} \\ 
       		&=\{(\omega\otimes\Id)(m\otimes \hat{b})
			\mid m\in\Comp(\Hils[L]), \omega\in\Bound(\Hils[L])_{*}, \hat{b}\in \Contvin(\dual{\qE})\}^{\textup{CLS}} \\
         	&= \Contvin(\dual{\qE}).
	\end{align*}  

 	\begin{proposition}
 		The \(\Cst\)\nb-algebra \(\dual{B}\) coincides with \(\Contvin(\dual{\qE})\).
 	\end{proposition}

 \begin{proof}  
   It is enough to show that \(N\) and \(\tilde{b}\) are affiliated with \(\dual{B}\). 
   The braided pentagon equation~\eqref{eq:DuBraided-pentagon} of \(\DuBrMultunit\) implies that
   \begin{equation}
   \DuBrMultunit^{*}_{12}\DuBrMultunit_{23}\DuBrMultunit_{12}\DuBrMultunit^{*}_{23}=
   \Dualbraiding{}{}_{23}\DuBrMultunit_{12}\Dualbraiding{}{}_{23}^*.
   \end{equation}
   The unitary \(\DuCorep{Y}\) commutes with \(\hat{Z}\), and we can infer from 
   the proof of Theorem~\ref{the:DuMultunit-main} that
   \begin{equation}
   \hat{Z}(1\otimes q^{\frac{\tilde{N}}{2}}b)\hat{Z}^* = P\otimes q^{\frac{\tilde{N}}{2}}b.
   \end{equation}
   Using this observation we compute
   \begin{equation*}
     \begin{aligned}
       \Dualbraiding{}{}_{23}\DuBrMultunit_{12}\Dualbraiding{}{}_{23}^*
       &=\Flip_{23} Z_{23}^*\QuantExp(Pb^{-1}q^{\frac{\tilde{N}}{2}}\otimes q^{\frac{\tilde{N}}{2}}b)_{12}^*\DuCorep{Y}_{12}Z_{23}\Flip_{23}\\
       &=  \hat{Z}_{23}\QuantExp(Pb^{-1}q^{\frac{\tilde{N}}{2}}\otimes  q^{\frac{\tilde{N}}{2}}b)_{13}^*\DuCorep{Y}_{13}
       		\hat{Z}_{23}^{*}\\
       &= \QuantExp(Pb^{-1}q^{\frac{\tilde{N}}{2}}\otimes P^* \otimes  q^{\frac{\tilde{N}}{2}}b)^*\DuCorep{Y}_{13}.
     \end{aligned}
   \end{equation*}
   Let \(\DuBrMultunit^{\lambda}
   	   =\QuantExp({\lambda}Pb^{-1}q^{\frac{\tilde{N}}{2}}\otimes q^{\frac{\tilde{N}}{2}}b)^{*}\DuCorep{Y}\) 
   and 
   \(S'(\lambda)=\QuantExp(\lambda Pb^{-1}q^{\frac{\tilde{N}}{2}}\otimes P^* \otimes q^{\frac{\tilde{N}}{2}}b)^*\DuCorep{Y}_{13}\). Then using the above analysis we have,
  \begin{equation}
    \label{eq:Apr27-9:38}
     (\DuBrMultunit^{\lambda})^{*}_{12} \DuBrMultunit_{23}\DuBrMultunit^{\lambda}_{12}\DuBrMultunit^{*}_{23}
     =\Dualbraiding{}{}_{23}\DuBrMultunit^{\lambda}_{12}\Dualbraiding{}{}_{23}^*
     = S'(\lambda)
  \end{equation}
	It follows that since the expression on the left hand side belongs to 
	\(\U(\Comp(\Hils[L])\otimes\Comp(\Hils[L])\otimes \dual{B})\), 
	\(\QuantExp(\lambda Pb^{-1}q^{\frac{\tilde{N}}{2}}\otimes P^* \otimes q^{\frac{\tilde{N}}{2}}b)^*\DuCorep{Y}_{13}\) 
	must also belong to \(\U(\Comp(\Hils[L])\otimes\Comp(\Hils[L])\otimes \dual{B})\) for all \(\lambda \in \Specn\).
 	Now \(\DuCorep{Y}\in\U(\Hils[L]\otimes \Hils[L])\) and as already observed
	\(\QuantExp(Pb^{-1}q^{\frac{\tilde{N}}{2}}\otimes q^{\frac{\tilde{N}}{2}}b)
	\in\U(\Comp(\Hils[L])\otimes \Contvin(\dual{\qE}))\), \(P\) is unitary implies that \(P\Aff\Comp(\Hils[L])\).
	Hence,  \(S'(\lambda)\) belongs to 
	\(\U(\Comp(\Hils[L])\otimes\Comp(\Hils[L])\otimes \dual{B})\) for all \(\lambda \in \Specn\). 
	Now, \(\{S'(\lambda)\}_{\lambda\in\Specn}\) constitute a strictly continuous family 
	of elements of \(\U(\Comp(\Hils[L])\otimes\Comp(\Hils[L])\otimes \dual{B})\) due to strict continuity of the map 
	\(\Specn\ni\lambda\to\QuantExp(\lambda Pb^{-1}q^{\frac{\tilde{N}}{2}}\otimes q^{\frac{\tilde{N}}{2}}b)
	\in\Mult(\Comp(\Hils[L])\otimes\dual{B})\).
	Therefore by~\cite[Proposition 5.2 ]{W1992a} we see 
	\(Pb^{-1}q^{\frac{\tilde{N}}{2}}\otimes P^{*}\otimes q^{\frac{\tilde{N}}{2}}b\  
	\Aff\  \Comp(\Hils[L])\otimes\Comp(\Hils[L])\otimes \dual{B}\). 
	Hence, \(q^{\frac{\tilde{N}}{2}}b\  \Aff\  \dual{B}\).

	Finally, note that \(S'(0)=\DuCorep{Y}_{13}\). Hence, \(\DuCorep{Y}\in\U(\Comp(\Hils[L])\otimes \dual{B})\). 
	Since \(\DuCorep{Y}\) generates \(\Z\) as a quantum group, the left slices 
	\(\{(\omega\otimes \Id_{\Comp(\Hils[L])})\DuCorep{Y}\}\) for 
	\(\omega\in\Bound(\Hils[L])_{*}\) are dense in \(\Contvin(\Z)\). 
	Moreover, by definition, \(\Sp(N)=\Z\) and \(\Contvin(\Z)\) is generated by \(N\Aff \Contvin(\Z)\). 
	Thus, we conclude that~\(N\Aff\dual{B}\) and this completes the proof.
  \end{proof}

\subsection{The Comultiplication map \(\DuComult[\qE]\)}
     \label{sec:Du-Eq2-comult} 
     In this section we provide the quantum group structure to the \(\Cst\)\nb-algebra constructed above 
     inside the category \(\Cstcat(\T)\).
     Let us recall the quantum group
     \(\T=(\Cont(\T),\Comult[\T])\), constructed from the manageable multiplicative unitary
      \(\Corep{W}\) on \(\Hils\otimes\Hils\) defined by \(\Corep{W}e_i\otimes e_j=e_i+e_{i+j}\), 
      see~\cite[Section 4]{RR2021},  can be viewed as a quasitriangular quantum group with respect to 
      the \(\Rmattxt\)\nb-matrix \(\Rmat\) that appears in section~\ref{sec:prelim}. 
      Also, we can readily see that  
     \(\Corep{U}=\Corep{W}_{23}\) acts on \(\Hils\otimes\Hils\otimes\Hils\cong\Hils[L]\otimes \Hils[H]\) by 
     \(\Corep{U}e_{i,j}\otimes e_k=e_{i,j}\otimes e_{j+k}\).
     
     A simple computation reveals that \(\Corep{U}(N\otimes 1)\Corep{U}^*=N\otimes 1\) and 
     \(\Corep{U}(\tilde{b}\otimes 1)\Corep{U}^*=\tilde{b}\otimes z^*\). Hence,
     the map \(\dual{\Gamma}\colon \Contvin(\dual{\qE})\to \Contvin(\dual{\qE})\otimes \Cont(\T)\) given by
     \[
     \dual{\Gamma}(N)= N\otimes 1, \qquad \dual{\Gamma}(\tilde{b})= \tilde{b}\otimes z^*
     \]
     yields a well defined action of \(\T\) on \(\Contvin(\dual{\qE})\). Therefore,
     \(\Contvin(\dual{\qE})\) is an object in \(\Cstcat(\T)\).
     
      Consider the monoidal product \(\boxtimes_{\DuRmat}\) on \(\Cstcat(\T)\) induced by the dual braiding 
      in~\(\Corepcat(\T)\) with respect to the dual \(\Rmattxt\)\nb-matrix \(\DuRmat\) and
      define the twofold twisted tensor product 
      \[
        \Contvin(\dual{\qE})\boxtimes_{\DuRmat}\Contvin(\dual{\qE})
        \defeq 
        (\Contvin(\dual{\qE})\otimes 1)\dual{Z}(1\otimes \Contvin(\dual{\qE}))\dual{Z}^*.
      \]
      The canonical embeddings 
      \(
        \dual{j_1},\dual{j_2}\in\Mor^{\T} (\Contvin(\dual{\qE}),\Contvin(\dual{\qE})\boxtimes_{\DuRmat}\Contvin(\dual{\qE}))
      \) 
      into the first and second factor of the twisted tensor product are given by:
       	\begin{equation*}
	  \dual{j_1}(\hat{b})=\hat{b}\otimes 1,
	  \quad
	  \dual{j_2}(\hat{b})=\dual{Z}(1\otimes \hat{b})\dual{Z}^*
	  \quad \textup{ for all } \hat{b}\in\Contvin(\dual{\qE}).
	\end{equation*}
    	Then, explicitly, on the distinguished generators~\(N\) and \(\tilde{b}\), 
	the embeddings  \(\dual{j_1}\) and \(\dual{j_2}\) are expressed as
	\begin{align*}
	    \dual{j_1}(N) &=N\otimes 1, 
	    &\quad 
	    \dual{j_2}(N) &=1\otimes N, \\
	    \dual{j_1}(\tilde{b}) &=\tilde{b}\otimes 1, 
	    &\quad \dual{j_2}(\tilde{b}) &=P\otimes \tilde{b}.
	\end{align*}

	Define \(\DuComult[\qE](\hat{b})=\DuBrMultunit(\hat{b}\otimes 1)\DuBrMultunit^*\) 
	for all \(\hat{b}\in \Contvin(\dual{\qE})\). 
	Then it is sufficient to compute the action of \(\DuComult[\qE]\) on the distinguished generators 
	of \(\Contvin(\dual{\qE})\): \( N\) and \(\tilde{b}=q^{\frac{\tilde{N}}{2}}b\). 
	We have already observed that \(\Sp(\hat{X})=\Specn\) and  
	have ~\(\QuantExp(\hat{X})\in\U(\Hils[L]\otimes\Hils[L])\).

 	Firstly, we observe that \(\DuComult[\qE](N)=\DuBrMultunit(N\otimes 1) \DuBrMultunit^{*}\). 
    	 Now,  \(\DuCorep{Y}(N\otimes 1)\DuCorep{Y}^*=N\otimes1\dotplus 1\otimes N\), 
	 where \(\dotplus\) denotes the closure of the sum of two normal operators. Therefore,
     	\begin{equation}
	  \label{eq:DuComult-on-N}
       	    \DuComult[\qE](N) = \QuantExp(\hat{X})^{*}(N\otimes 1\dotplus 1\otimes N)\QuantExp(\hat{X})
 	\end{equation}
	The commutation relations \(Nb=b(N-2I)\) and \(Nb^{-1}=b^{-1}(N+2I)\) together imply that 
	\(N\otimes 1+1\otimes N\) commutes with
	\(\hat{X}\). Consequently, the equation \eqref{eq:DuComult-on-N} reduces to
	\[
	  \DuComult[\qE](N) = N\otimes 1\dotplus 1\otimes N= \dual{j_1}(N) \dotplus \dual{j_2}(N).
	\]
	Secondly, again using the definitions, we have 
	\(\DuCorep{Y}(q^{\frac{\tilde{N}}{2}}b\otimes 1)\DuCorep{Y}^*=q^{\frac{\tilde{N}}{2}}b\otimes 1\).
	Therefore, substituting \(\hat{b}=\tilde{b}\) in the formula for \(\DuComult[\qE]\), we have
	\begin{equation}
  	  \label{eq:DuComult-on-tilde-b}
	    \DuComult[\qE](\tilde{b})=\DuBrMultunit(q^{\frac{\tilde{N}}{2}}b\otimes 1)\DuBrMultunit^*
	    =\QuantExp(\hat{X})^*(q^{\frac{\tilde{N}}{2}}b\otimes 1)\QuantExp(\hat{X}).
	\end{equation}
	Let~\(\Upsilon\) be a closed operator on \(\Hils[L]\) such that \(\ker(\Upsilon)=\{0\}\), the spectrum \(\Sp(\Upsilon)\subset\Specn\)
	and \(\Ph{\Upsilon}\Mod{\Upsilon}\Ph{\Upsilon}^*=q^{-1}\Mod{\Upsilon}\) where 
	\(\Upsilon=\Ph{\Upsilon}\Mod{\Upsilon}\) is the polar decomposition of \(\Upsilon\).
	Let \(R=\Upsilon\otimes b^*(\conj{q})^{\frac{\tilde{N}}{2}} \otimes 1\) and
 	\(S= \Upsilon \otimes (\conj{q})^{N}P^*\otimes b^*(\conj{q})^{\frac{\tilde{N}}{2}}\). Then, a simple computation shows that \(R\) and \(S\) 
	are normal operators having their spectra contained in \(\Specn\) and satisfy the relations in \cite{W1992a}*{(0.1)}. Now
	\begin{align*}
	    R^{-1}S &= (\Upsilon^{-1}\otimes (\conj{q})^{-\frac{\tilde{N}}{2}}(b^*)^{-1}\otimes 1 )
				(\Upsilon \otimes (\conj{q})^{N}P^*\otimes b^*(\conj{q})^{\frac{\tilde{N}}{2}})\\
			&=1\otimes (\conj{q})^{-\frac{\tilde{N}}{2}}(b^*)^{-1}(\conj{q})^{N}P^*\otimes b^*(\conj{q})^{\frac{\tilde{N}}{2}}\\
			&=1\otimes (\conj{q})^{-\frac{\tilde{N}}{2}}(q^N b^{-1})^*P^*\otimes b^*(\conj{q})^{\frac{\tilde{N}}{2}}\\
			&= 1\otimes (\conj{q})^{-\frac{\tilde{N}}{2}}( q^2 b^{-1}q^N)^*P^*\otimes b^*(\conj{q})^{\frac{\tilde{N}}{2}}\\
			&= 1\otimes (\conj{q})^{\frac{\tilde{N}}{2}}( b^{-1})^*P^*\otimes b^*(\conj{q})^{\frac{\tilde{N}}{2}}\\
			&=\hat{X}_{23}^*.
	\end{align*}
	In the above computation we use the commutation relation \(q^N b^{-1}=q^2 b^{-1}q^N\), because 
	\(q^N b^{-1}e_{i,j}=q^2 q^{i+j+\frac{i-j}{2}}e_{i+1,j+1}=q^2b^{-1}q^Ne_{i,j}\).
	Replacing \(\tilde{b}\) by \(\tilde{b}^*\) in equation \eqref{eq:DuComult-on-tilde-b} and 
	applying~\cite[Theorem 2.2 and Theorem 3.1]{W1992a} we have
	\[
	  \QuantExp(\hat{X})_{23}^*(\Upsilon \otimes b^*(\conj{q})^{\frac{\tilde{N}}{2}}\otimes 1)\QuantExp(\hat{X})_{23}
	  =\Upsilon\otimes b^*(\conj{q})^{\frac{\tilde{N}}{2}} \otimes 1
  	  \dotplus \Upsilon \otimes (\conj{q})^{N}P^*\otimes b^*(\conj{q})^{\frac{\tilde{N}}{2}}.
	\]
	Since \(\Upsilon\) is chosen arbitrarily, we get that
	\[
	  \DuComult[\qE](\tilde{b}^*) = b^*(\conj{q})^{\frac{\tilde{N}}{2}} \otimes 1
							\dotplus (\conj{q})^{N}P^*\otimes b^*(\conj{q})^{\frac{\tilde{N}}{2}}
					 	   =  \dual{j_1}(\tilde{b})^* \dotplus  \dual{j_1}(q^{N})^* \dual{j_2}(\tilde{b})^*.
	\]
	Coassociativity and cancellation conditions for \(\DuComult[\qE]\) follows from 
	Theorem~\ref{the:DuBQgrp-quasitriag}.
	In summary we have,
	
     \begin{theorem}
        \label{the:Du-qE-Comult}
       	There exists \(\DuComult[\qE]\in\Mor^{\T}(\Contvin(\dual{\qE}),\Contvin(\dual{\qE})\boxtimes_{\DuRmat}
	\Contvin(\dual{\qE}))\) 
 	such that
   	\begin{equation*}
     	  \DuComult[\qE](N) = \dual{j_1}(N) \dotplus \dual{j_2}(N) 
     	  \quad \text{ and }\quad
     	  \DuComult[\qE](\tilde{b}) = \dual{j_1}(\tilde{b}) \dotplus  \dual{j_1}(q^{N}) \dual{j_2}(\tilde{b}).
   	\end{equation*}
   	Moreover, \(\DuComult[\qE]\) satisfies the coassociativity condition: 
   	\[
      	   (\Id_{\Contvin(\dual{\qE})}\boxtimes_{\DuRmat}\DuComult[\qE])\circ \DuComult[\qE]
      	    =(\DuComult[\qE]\boxtimes_{\DuRmat}\Id_{\Contvin(\dual{\qE})})\circ\DuComult[\qE];
   	\]
    	and the cancellation condition:
   	\begin{align*}
     	  \dual{j_1}(\Contvin(\dual{\qE}))(1_{\Contvin(\dual{\qE})}\boxtimes_{\DuRmat}\Contvin(\dual{\qE}))
     	   &=\Contvin(\dual{\qE})\boxtimes_{\DuRmat}\Contvin(\dual{\qE})\\
     	  &=(\Contvin(\dual{\qE})\boxtimes_{\DuRmat} 1_{\Contvin(\dual{\qE})})\dual{j_2}(\Contvin(\dual{\qE})).
   	\end{align*}
   \end{theorem}
   	Finally combining Proposition~\ref{prop:Du-qE-Cstalg} and Theorem~\ref{the:Du-qE-Comult} 
	we obtain the dual of braided quantum \(\textup{E}(2)\) group.
	
\begin{theorem}
   	The triple \((\Contvin(\dual{\qE}),\DuComult[\qE],\dual{\Gamma})\) is a braided \(\Cst\)\nb-quantum group 
	over~\(\T\) with respect to \((\corep{U},\DuRmat)\) and the dual of~\(\dual{\qE}\) is isomorphic to the braided 
	\(\qE\) over \(\T\) with respect to~\((\corep{U},\Rmat)\).
\end{theorem}
 
\section{Bosonization}
  \label{sec:Bosonization} 

The concept of combining an ordinary quantum group and a braided quantum group in a systematic manner to create a new ordinary quantum group, known as the bosonization construction, was introduced in~\cite{MRW2016}. This construction is rooted in the algebraic Radford bosonization~\cite{Rad1985}, which, in turn, relies on the semidirect product construction for groups. It establishes an equivalence between braided compact quantum groups and a specific category of ordinary quantum groups with an idempotent quantum group morphism. In this section, we present a detailed exposition of the bosonization process for the dual braided quantum group constructed in the previous section.

	We construct an ordinary quantum group \(\dual{\qE}\rtimes_{\dual{\Gamma}}\T\) with an idempotent quantum group homomorphism with image \(\T\). Therefore, \(\dual{\qE}\rtimes_{\dual{\Gamma}}\T\) is the analytic counterpart of the bosonization of \(\dual{\qE}\), which is a new example of (non-compact) locally compact quantum group.

	The \(\Rmattxt\)\nb-matrix \(\Rmat\) 
	induces a representation~\(\corep{V}'\in\U(\Comp(\Hils[L])\otimes\Contvin(\Z))\) satisfying 
	equation~\((32)\)~\cite[Theorem 5.3]{MRW2012}. Denoting \(\corep{V}'\) as \(\Corep{V}'\) 
	when viewed as an element of 
	\(\U(\Hils[L]\otimes \Hils)\) define \(\DuCorep{V}'=\Flip\Corep{V}'^{*}\Flip\).
	Recall the action of \(\Corep{W}\) on \(\Hils\otimes \Hils\), on the basis elements, is given by
	\(\Corep{W} e_{p}\otimes e_{k} = e_{p}\otimes e_{k+p}\) and  
	\(\DuCorep{V}'\) acts on \(\Hils\otimes\Hils\otimes\Hils\cong\Hils\otimes\Hils[L]\) by
	\(\DuCorep{V}' e_p\otimes e_{i,j}=\zeta^{-pj}e_p\otimes e_{i,j}\). 
	Then according to~\cite[Theorem 3.8]{MRW2017}
	\begin{equation}
	  \label{def:Multunit-boson}
		\widetilde{\mathcal{W}}=\Corep{W}_{13}\Corep{U}_{23}\DuCorep{V}_{34}'^*\DuBrMultunit_{24}\DuCorep{V}'_{34} 
		\quad \textup{ in } \U(\Hils[H]\otimes\Hils[L]\otimes\Hils[H]\otimes\Hils[L])
	\end{equation}
	is a manageable multiplicative unitary. It is known that every manageable 
	multiplicative unitary corresponds to a
	\(\Cst\)\nb-quantum group~\cite{W1996}.
	Let \(\widetilde{\mathcal{W}}\) generate the \(\Cst\)\nb-quantum group \(\dual{\qE}\rtimes_{\dual{\Gamma}}\T\). 
	It is a semidirect product quantum group or a quantum group with projection in the sense of~\cite{MRW2017}, 
	hence explains our notation.
	Observe that \(\DuCorep{Y}=\DuCorep{W}_{13}\DuCorep{W}_{14}\) on \(\Hils^{\otimes 4}\)
	and \(\Corep{U}=\Corep{W}_{23}\). Hence, equation \eqref{def:Multunit-boson} can be written as
	\begin{equation}
	  \label{eq:Aux-multunit-boson}
		\widetilde{\mathcal{W}}=\Corep{W}_{14}\Corep{W}_{34}\DuCorep{V}_{456}'^*
                     \QuantExp(\hat{X})_{2356}^*
                     \DuCorep{W}_{25}\DuCorep{W}_{26}
                     \DuCorep{V}_{456}' \quad \textup{ in } \U(\Hils^{\otimes 6}).
	\end{equation}
	The \(\Cst\)\nb-algebra~\(\Contvin(\dual{\qE}\rtimes_{\dual{\Gamma}}\T)
		\defeq 
		\Cont(\T)\boxtimes_{\DuRmat}\Contvin(\dual{\qE})=
		j_{\T}(\Cont(\T))j_{\dual{\qE}}(\Contvin(\dual{\qE}))\) is a subalgebra of \(\Bound(\Hils\otimes\Hils[L])\), where 
		\(j_{\T}(z)=z\otimes 1\) and 
  		\(j_{\dual{\qE}}(\hat{b})=\DuCorep{V}'^{*}(1\otimes \hat{b})\DuCorep{V}'\), 
		\(z\) is the unitary generator of~\(\Cont(\T)\) and~\(\hat{b}\in \Contvin(\dual{\qE})\subset\Bound(\Hils[L])\). 
	Then a short calculation using the definitions of~\(\tilde{b}, N, z\) yields
	\begin{align*}
	   j_{\dual{\qE}}(N) &= 1\otimes N,
	   &\quad
	   j_{\T}(z)j_{\dual{\qE}}(N) &= j_{\dual{\qE}}(N)j_{\T}(z),\\
	   j_{\dual{\qE}}(\tilde{b}) &= P'\otimes \tilde{b},
	   &\quad
	   j_{\T}(z)j_{\dual{\qE}}(\tilde{b}) &= \zeta j_{\dual{\qE}}(\tilde{b})j_{\T}(z).
	\end{align*}
	The comultiplicaton \(\Comult[{\dual{\qE}\rtimes_{\dual{\Gamma}}\T}]
	  \colon\Contvin(\dual{\qE}\rtimes_{\dual{\Gamma}}\T)\to
	  \Mult(\Contvin(\dual{\qE}\rtimes_{\dual{\Gamma}}\T)\otimes \Contvin(\dual{\qE}\rtimes_{\dual{\Gamma}}\T))\) 
	  is given by
	\[
	  \Comult[{\dual{\qE}\rtimes_{\dual{\Gamma}}\T}](d)=
	    \widetilde{\mathcal{W}}(d\otimes 1)\widetilde{\mathcal{W}}^* 
	     \quad\text{ for all } d\in \Contvin(\dual{\qE}\rtimes_{\dual{\Gamma}}\T).
	\]

	Notice that \(\DuCorep{W}_{25}\DuCorep{W}_{26}\) acts trivially on the fifth and sixth tensor factor. 
	Hence, it commutes with \(\DuCorep{V}_{456}'\). Also, the commutation relation 
	\(\DuCorep{V}'^*(1 \otimes q^{\frac{\tilde{N}}{2}}b)\DuCorep{V}'=P'\otimes q^{\frac{\tilde{N}}{2}}b\)
 allows us to rewrite equation \eqref{eq:Aux-multunit-boson} as 
\begin{equation*}
\widetilde{\mathcal{W}}=\Corep{W}_{14}\Corep{W}_{34}
                     \QuantExp(Pb^{-1}q^{\frac{\tilde{N}}{2}} \otimes P'\otimes q^{\frac{\tilde{N}}{2}}b)_{23456}^*
                     \DuCorep{W}_{25}\DuCorep{W}_{26}
                     \quad \textup{ in } \U(\Hils^{\otimes 6}).
\end{equation*}
Moreover,
 \(\Corep{W}_{14}(z\otimes 1_{\Hils^{\otimes 2}}\otimes 1_{\Hils^{\otimes 3}})\Corep{W}_{14}^*
= z\otimes 1_{\Hils^{\otimes 2}}\otimes z\otimes 1_{\Hils^{\otimes 2}}=j_{\T}(z)\otimes j_{\T}(z)\).
Then clearly, for \(d=j_{\T}(z)=z\otimes 1\), we obtain
\[
\Comult[\dual{\qE}\rtimes_{\dual{\Gamma}}\T](j_{\T}(z)) = j_{\T}(z)\otimes j_{\T}(z).
\]
Now for \(d=j_{\dual{\qE}}(N)\) we first see that
\begin{align*}
&\DuCorep{W}_{25}\DuCorep{W}_{26}(1\otimes N\otimes 1_{\Hils^{\otimes 3}})
    \DuCorep{W}_{26}^*\DuCorep{W}_{25}^* e_p\otimes e_{ij}\otimes e_q\otimes e_{mn}\\
&=\DuCorep{W}_{25}\DuCorep{W}_{26}(1\otimes N\otimes 1_{\Hils^{\otimes 3}})
      e_p\otimes e_{i+m+n,j}\otimes e_q\otimes e_{mn}\\
&=\DuCorep{W}_{25}\DuCorep{W}_{26}
      e_p\otimes (i+j+m+n)e_{i+m+n,j}\otimes e_q\otimes e_{mn}\\
&=e_p\otimes (i+j)e_{i,j}\otimes e_q\otimes e_{mn} + e_p\otimes e_{i,j}\otimes e_q\otimes (m+n)e_{mn}\\
&= (1\otimes N\otimes 1_{\Hils^{\otimes 3}} + 1_{\Hils^{\otimes 3}}\otimes 1\otimes N)
      e_p\otimes e_{ij}\otimes e_q\otimes e_{mn}.
\end{align*}
We express \(\Comult[{\dual{\qE}\rtimes_{\dual{\Gamma}}\T}]\) as \(\Comult[{\dual{\qE}\rtimes_{\dual{\Gamma}}\T}](j_{\dual{\qE}}(N))=
 \Comult[{\dual{\qE}\rtimes_{\dual{\Gamma}}\T}]^1 +  \Comult[{\dual{\qE}\rtimes_{\dual{\Gamma}}\T}]^2\) where 
\begin{multline}
\Comult[{\dual{\qE}\rtimes_{\dual{\Gamma}}\T}]^1 =\Corep{W}_{14}\Corep{W}_{34}
                     	\QuantExp(Pb^{-1}q^{\frac{\tilde{N}}{2}} \otimes P'\otimes q^{\frac{\tilde{N}}{2}}b)_{23456}^*
			(1\otimes N\otimes 1_{\Hils^{\otimes 3}})\\
			\QuantExp(Pb^{-1}q^{\frac{\tilde{N}}{2}} \otimes P'\otimes q^{\frac{\tilde{N}}{2}}b)_{23456}
			\Corep{W}_{34}^*\Corep{W}_{14}^*, 
\end{multline}

\begin{multline}
\Comult[{\dual{\qE}\rtimes_{\dual{\Gamma}}\T}]^2 =\Corep{W}_{14}\Corep{W}_{34}
                     	\QuantExp(Pb^{-1}q^{\frac{\tilde{N}}{2}} \otimes P'\otimes q^{\frac{\tilde{N}}{2}}b)_{23456}^*
			(1_{\Hils^{\otimes 3}}\otimes 1\otimes N)\\
			\QuantExp(Pb^{-1}q^{\frac{\tilde{N}}{2}} \otimes P'\otimes q^{\frac{\tilde{N}}{2}}b)_{23456}
			\Corep{W}_{34}^*\Corep{W}_{14}^*.
\end{multline}
Observe that \(Nb^{-1}=b^{-1}(N+2I)\) and \(Nb=b(N-2I)\). Therefore the above equations reduce to
\begin{align*}
\Comult[{\dual{\qE}\rtimes_{\dual{\Gamma}}\T}]^1 
	&= \Corep{W}_{14}\Corep{W}_{34}(1\otimes (N+2)\otimes 1_{\Hils^{\otimes 3}})\Corep{W}_{34}^*\Corep{W}_{14}^*
	  =1\otimes N\otimes 1_{\Hils^{\otimes 3}} + 2\cdot 1_{\Hils^{\otimes 6}},\\
\Comult[{\dual{\qE}\rtimes_{\dual{\Gamma}}\T}]^2 
	&= \Corep{W}_{14}\Corep{W}_{34}(1_{\Hils^{\otimes 3}}\otimes 1\otimes (N-2))\Corep{W}_{34}^*\Corep{W}_{14}^*
	  = 1_{\Hils^{\otimes 3}}\otimes 1\otimes N - 2\cdot 1_{\Hils^{\otimes 6}}.
\end{align*}
Summing up we have,
\begin{equation*}
\Comult[{\dual{\qE}\rtimes_{\dual{\Gamma}}\T}](j_{\dual{\qE}}(N)) =j_{\dual{\qE}}(N)\otimes 1 \dotplus 1\otimes j_{\dual{\qE}}(N).
\end{equation*}

Now we compute \(\Comult[{\dual{\qE}\rtimes_{\dual{\Gamma}}\T}](d)\) for \(d=j_{\dual{\qE}}(\tilde{b}^*)\).  We choose to calculate the action of the comultiplication map on \(\tilde{b}^*\) instead of \(\tilde{b}\) to avoid some computational hurdles.

By definition 
\(\Comult[{\dual{\qE}\rtimes_{\dual{\Gamma}}\T}](j_{\dual{\qE}}(\tilde{b})^*)=\widetilde{\mathcal{W}}(P'\otimes q^{\frac{\tilde{N}}{2}}b\otimes 1_{\Hils^{\otimes 3}})^*\widetilde{\mathcal{W}}^*\). By analogous computation as above we see that 
\(\DuCorep{W}_{26}\DuCorep{W}_{25}\) commutes with \((P'\otimes q^{\frac{\tilde{N}}{2}}b\otimes 1_{\Hils^{\otimes 3}})\).
Therefore we have 
\begin{multline}
\label{eq:boson-comult-b}
\Comult[{\dual{\qE}\rtimes_{\dual{\Gamma}}\T}](j_{\dual{\qE}}(\tilde{b}^*)) 
     = \Corep{W}_{14}\Corep{W}_{34}
          \QuantExp(Pb^{-1}q^{\frac{\tilde{N}}{2}} \otimes P'\otimes q^{\frac{\tilde{N}}{2}}b)_{23456}^*
	(P'\otimes q^{\frac{\tilde{N}}{2}}b\otimes 1_{\Hils^{\otimes 3}})^* \\
	\QuantExp(Pb^{-1}q^{\frac{\tilde{N}}{2}} \otimes P'\otimes q^{\frac{\tilde{N}}{2}}b)_{23456}
	\Corep{W}_{34}^*\Corep{W}_{14}^*.
\end{multline}
Now consider \(R=P'^{*}\otimes b^*(\conj{q})^{\frac{\tilde{N}}{2}}\otimes 1_{\Hils^{\otimes 3}}\) 
and \(S=P'^{*}\otimes (\conj{q})^{N}P^*\otimes P'^{*}\otimes b^*(\conj{q})^{\frac{\tilde{N}}{2}}\). Then 
\[
R^{-1}S=1\otimes (\conj{q})^{\frac{\tilde{N}}{2}}(b^*)^{-1}P^*\otimes P'^{*}\otimes  b^*(\conj{q})^{\frac{\tilde{N}}{2}}.
\]
Then again invoking~\cite[Theorem 2.2 and Theorem 3.1]{W1992a} equation \eqref{eq:boson-comult-b} becomes
\begin{align}
\label{eq:comult-bos-intermediate}
   &\Comult[{\dual{\qE}\rtimes_{\dual{\Gamma}}\T}](j_{\dual{\qE}}(\tilde{b}^*)) \nonumber\\
   &\quad=\Corep{W}_{14}\Corep{W}_{34}
(P'^{*}\otimes b^*(\conj{q})^{\frac{N}{2}}\otimes 1_{\Hils^{\otimes 3}}\dotplus P'^{*}\otimes (\conj{q})^{N}P^*\otimes P'^{*}\otimes b^*(\conj{q})^{\frac{N}{2}})
\Corep{W}_{34}^*\Corep{W}_{14}^*.
\end{align}
We compute each term on the right hand side of the above equation separately. Note that \(\Corep{W}_{14}\Corep{W}_{34}\) effects a change only in the fourth leg, hence it commutes with 
\(P'^{*}\otimes b^*(\conj{q})^{\frac{N}{2}}\otimes 1_{\Hils^{\otimes 3}}\). Therefore the first term remains
\(P'^{*}\otimes b^*(\conj{q})^{\frac{N}{2}}\otimes 1_{\Hils^{\otimes 3}}\). As for the second term, we have
\begin{align*}
& \Corep{W}_{14}\Corep{W}_{34} (P'^{*}\otimes (\conj{q})^{N}P^*\otimes P'^{*})\Corep{W}_{34}^*\Corep{W}_{14}^* e_p\otimes e_{ij}\otimes e_k \\
&=\Corep{W}_{14}\Corep{W}_{34} (P'^{*}\otimes (\conj{q})^{N}P^*\otimes P'^{*}) e_p\otimes e_{ij}\otimes e_{k-p-j} \\
&=\Corep{W}_{14}\Corep{W}_{34} \zeta^{p}e_p\otimes \zeta^{j}\conj{q}^{i+j}e_{ij}\otimes \zeta^{k-p-j}e_{k-p-j} \\
&= e_p\otimes \conj{q}^{i+j}e_{ij}\otimes \zeta^{k}e_{k}\\
&=(1\otimes (\conj{q})^{N}\otimes P'^{*})e_p\otimes e_{ij}\otimes e_k .
\end{align*}
Therefore the second term in the right hand side of \eqref{eq:comult-bos-intermediate} becomes \(1\otimes (\conj{q})^{N}\otimes P'^{*}\otimes b^*(\conj{q})^{\frac{N}{2}}\). 
Thus we finally obtain
\[
\Comult[{\dual{\qE}\rtimes_{\dual{\Gamma}}\T}](j_{\dual{\qE}}(\tilde{b}^*))
	=j_{\dual{\qE}}(\tilde{b}^*) \otimes 1
	\dotplus
	j_{\dual{\qE}}((\conj{q})^{N})\otimes j_{\dual{\qE}}(\tilde{b}^*).	
\]

We sum up the above calculations in the following theorem yielding the bosonization of the dual braided quantum \(\textup{E}(2)\) group.
\begin{theorem}
\(\Contvin(\dual{\qE}\rtimes_{\dual{\Gamma}}\T)\) is the universal \(\Cst\)\nb-algebra generated by a unitary \(u\), a self-adjoint operator \(N'\) with integer spectrum and an operator \(b'\) with \(\Sp(b')=\Mod{q}^\Z\) subject to the commutation relations
	  \[
	   uN'=N'u,\quad ub'=\zeta b'u,\quad N'b'=b'(N'-2I).
	  \]
	  There exists a unique 
	  \(\Comult[{\dual{\qE}\rtimes_{\dual{\Gamma}}\T}]\colon\Contvin(\dual{\qE}\rtimes_{\dual{\Gamma}}\T)\to \Mult(\Contvin(\dual{\qE}\rtimes_{\dual{\Gamma}}\T)\otimes \Contvin(\dual{\qE}\rtimes_{\dual{\Gamma}}\T))\) such that
	  \begin{align*}
	  \Comult[{\dual{\qE}\rtimes_{\dual{\Gamma}}\T}](u)&=u\otimes u,\\
	   \Comult[{\dual{\qE}\rtimes_{\dual{\Gamma}}\T}](N')&=N'\otimes 1 \dotplus 1\otimes N',\\
	  \Comult[{\dual{\qE}\rtimes_{\dual{\Gamma}}\T}](b')&=b'\otimes 1 \dotplus q^{N'}\otimes b',
	  \end{align*}
	  and~\((\Contvin(\dual{\qE}\rtimes_{\dual{\Gamma}}\T),\Comult[{\dual{\qE}\rtimes_{\dual{\Gamma}}\T}])\) is a \(\Cst\)\nb-quantum group. Moreover, there exists an
  	idempotent Hopf \Star{}homomorphism \(g\colon\Contvin(\dual{\qE}\rtimes_{\dual{\Gamma}}\T)\to\Mult(\Contvin(\dual{\qE}\rtimes_{\dual{\Gamma}}\T))\) with~\(g(u)=u\), 
 	\(g(N')=0\) and~\(g(b')=0\). Its image is the copy of \(\Cont(\T)\) generated 
 	by~\(u\) as a closed quantum subgroup of~\(\dual{\qE}\rtimes_{\dual{\Gamma}}\T\) and its kernel is the copy 
 	of~\(\Contvin(\dual{\qE})\) generated by~\(N',b'\) as the braided~\(\dual{\qE}\) group over~\(\T\).
\end{theorem}
 
\medskip


\end{document}